\documentclass[11pt,a4paper]{article}
\usepackage[utf8]{inputenc}
\usepackage[T1]{fontenc}
\usepackage{lmodern,amsmath,amssymb,amsthm,booktabs,microtype,needspace}
\usepackage{array,titlesec,graphicx}
\usepackage[margin=2.6cm]{geometry}
\usepackage[hidelinks]{hyperref}
\titleformat{\section}{\normalfont\Large}{\thesection}{0.7em}{}
\titleformat{\subsection}{\normalfont\large}{\thesubsection}{0.7em}{}
\newtheoremstyle{plainhead}{\topsep}{\topsep}{\itshape}{}{\normalfont}{.}{.5em}{}
\theoremstyle{plainhead}
\newtheorem{theorem}{Theorem}
\newtheorem{lemma}[theorem]{Lemma}
\newtheorem{proposition}[theorem]{Proposition}
\newtheorem{conjecture}[theorem]{Conjecture}
\newcommand{\E}{\mathbb E}
\graphicspath{{figures/}}
\hypersetup{pdftitle={Perturbing the Conway sequence},pdfauthor={Benoit Cloitre},
 pdfsubject={Version 8: final editorial revision and explicit algebraic weights}}
\title{Perturbing the Conway sequence}
\author{Benoit Cloitre}
\date{October 21, 2026}
\begin{document}
\raggedbottom
\maketitle

\begin{abstract}
We study the alternating perturbation of Conway's nested recurrence,
$C(n)=C(C(n-1))+C(n-C(n-1))+(-1)^n$ for $n\ge3$, with $C(1)=C(2)=1$.
We prove that the sequence is well defined and satisfies
$C(n)=\psi^{-3}n+O(n/\sqrt{\log n})$, where $\psi>1$ is the real root of
$\psi^3=\psi^2+1$.
\end{abstract}

\section{Introduction}\label{sec:main}
We consider the sequence (OEIS A399773~\cite{CloitreOEIS}) defined by
\begin{equation}\label{eq:C}
 C(1)=C(2)=1,\qquad
 C(n)=C(C(n-1))+C(n-C(n-1))+(-1)^n\quad(n\ge3).
\end{equation}
Without the alternating term, Conway's recurrence has limiting ratio
$1/2$; see Mallows~\cite{Mallows} and Kubo--Vakil~\cite{KuboVakil}.
The alternating perturbation changes this limiting ratio. Let $\psi>1$
be the real root of $\psi^3=\psi^2+1$, so that the candidate slope is
$\psi^{-3}=0.317672196\ldots$.
Throughout, $\log$ denotes the natural logarithm.

\begin{theorem}\label{thm:main}
The recurrence \eqref{eq:C} is well defined for every positive integer, and
\begin{equation}\label{eq:rate}
 C(n)=\frac{n}{\psi^3}+O\!\left(\frac{n}{\sqrt{\log n}}\right).
\end{equation}
In particular, $C(n)/n\longrightarrow\psi^{-3}$.
\end{theorem}

The motivation was Mantovanelli's addition of $(-1)^n$ to Hofstadter's
$Q$-recurrence~\cite{Mantovanelli}. In~\cite{CloitreQ}, we proved that the
resulting sequence $\widetilde Q$ is well defined and satisfies
$\widetilde Q(n)=n/2+O(n/\sqrt{\log n})$, with this error order optimal.
That proof uses binary words, Dyck paths, and plane forests. Here the
argument instead combines exact values on a geometric subsequence with
a quadratic estimate on the recursion tree.

\subsection{Motivation and proof outline}\label{sec:discovery}
The first terms are
\[
 C=1,1,1,3,1,3,3,3,3,5,3,5,3,5,5,5,\ldots.
\]
Although $C$ is not monotone, the graph of $C(n)/n$ reveals a regular
pattern of arches on a logarithmic horizontal scale
(Figure~\ref{fig:contacts}). Examining selected local maxima of
$C(n)-\psi^{-3}n$ led to two empirical families of indices,
\[
 (110,230,488,1042,2232,4788,\ldots),\qquad
 (704,1516,3260,7006,15052,\ldots).
\]
Starting at $488$, they interleave as
$488,704,1042,1516,2232,3260,\ldots$, with consecutive ratios appearing
to approach $\psi$. Since $\psi^3=\psi^2+1$, it is natural to test the
residual $T_j-T_{j-1}-T_{j-3}$. In the computed list, it alternated
between $-14$ and $12$, beginning at $j=3$. We therefore define the
reference indices by
\begin{equation}\label{eq:T}
 \begin{split}
 T_0&=488,\quad T_1=704,\quad T_2=1042,\\
 T_j&=T_{j-1}+T_{j-3}-1+13(-1)^j\qquad(j\ge3).
 \end{split}
\end{equation}
We shall prove exact formulas for $C(T_j)$ and $C(T_j-1)$ in
Section~\ref{sec:profiles}. In particular, $C(T_j)=\psi^{-3}T_j+O(1)$, so the
marked values of $C(T_j)/T_j$ approach $\psi^{-3}$ at order $T_j^{-1}$.

\begin{figure}[t]
\centering
\includegraphics[width=0.96\linewidth]{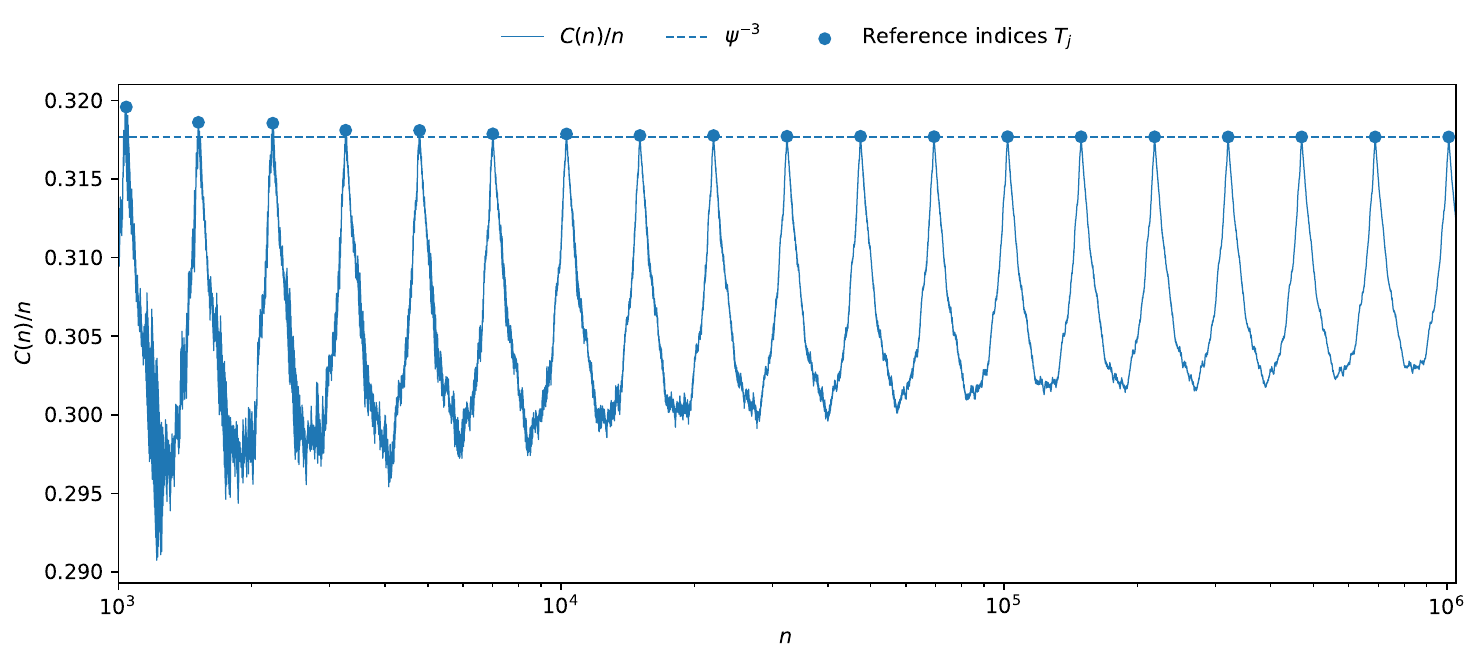}
\caption{$C(n)/n$ for $10^3\le n\le2^{20}$, with a logarithmic horizontal
axis. The markers show the computed values at the reference indices
$T_j$. The exact formulas proved in Section~\ref{sec:profiles} imply
$C(T_j)/T_j=\psi^{-3}+O(T_j^{-1})$.}
\label{fig:contacts}
\end{figure}

A comparison with the shifted sequence
\[
 R(1)=R(2)=R(3)=1,\qquad
 R(n)=R(R(n-2))+R(n-1-R(n-2))\quad(n\ge4)
\]
also helped identify the local structure. This sequence satisfies
$R(n)=\psi^{-3}n+O(1)$~\cite{CloitreM}. Jeffrey Shallit supplied a conjectural
automaton for $\max\{C(n)-R(n),0\}$ in Narayana
numeration~\cite{ShallitDIFF}. Its repeating patterns included the
reference families already observed and suggested examining short
windows around them. The windows led to the two tables used below;
their values are established directly from \eqref{eq:C}.

The proof has two steps. First, the odd and even positions of $C$ give
two nondecreasing sequences $A,B$ whose increments are $0$ or $1$.
Around each $s_j=T_j/2$, a short table of their values, translated
vertically by $s_{j-3}$, repeats according to the parity of $j$.
The recurrence reads neighboring entries as well as the centers, so
these tables contain just the extra information needed for induction. The main difficulty
is locating the left endpoint of the next window. At the first index
where $A$ reaches a value $v$, the preceding value is necessarily
$v-1$; this observation determines the endpoint from earlier tables.

The second step controls all indices between the reference points.
Their spacing is comparable to their size, so monotonicity alone
would leave a nonvanishing relative error. Expanding the recurrence
on a binary tree gives weighted averages, with weights $\psi^{-3},\psi^{-1}$
corresponding to descents of three and one reference levels.
The key comparison is between the paths $10$ and $01$: their weights
agree, but they end at different functions. This relates an error in
a sequence value to changes of the arguments in the tree. Summing
squared changes over $H$ levels gives an $O(H^{-1/2})$ bound. The exact
tables permit $H$ to be proportional to $\log n$, yielding
\eqref{eq:rate}.

Sections~\ref{sec:parities}--\ref{sec:profiles} establish the parity
properties, the first-passage lemma, and the local tables.
Section~\ref{sec:rate} gives the tree estimate and completes the proof.
The finite initial calculation is recorded in Appendix~\ref{app:seed}.

\section{Splitting the sequence}\label{sec:parities}
The perturbations at $n$ and $n+2$ have the same sign, so they cancel
when the corresponding recurrences are subtracted. This gives
monotonicity separately on the odd and even positions.

\begin{lemma}\label{lem:basic}
The recurrence \eqref{eq:C} is well defined. Its values are positive and odd,
and
\[
 1\le C(n)\le n-1\quad(n\ge2),\qquad
 C(n+2)-C(n)\in\{0,2\}\quad(n\ge1).
\]
Both parity subsequences are unbounded, and each takes every positive odd
integer value.
\end{lemma}
\begin{proof}
The values $C(1)=C(2)=C(3)=1$ start the induction for well-definedness.
Let $n\ge4$, assume the earlier values are positive and odd with the
stated bound, and put $u=C(n-1)$. Then $1\le u\le n-2$, so both
$u$ and $n-u$ are earlier positive indices. The recurrence gives a
positive odd value. If $u\ge2$, both arguments are at least $2$, and
\[
 C(n)\le(u-1)+(n-u-1)+1=n-1.
\]
If $u=1$, then $C(n-1)=1$ and
$C(n)=2+(-1)^n\le3\le n-1$.

For the parity increments, use a separate induction on $n$. The cases
$n=1,2$ give $C(3)-C(1)=0$ and $C(4)-C(2)=2$. For $n\ge3$, put
$u=C(n-1)$ and $v=C(n+1)$. The preceding step gives $v-u\in\{0,2\}$,
while subtraction yields
\[
 C(n+2)-C(n)=C(v)+C(n+2-v)-C(u)-C(n-u).
\]
If $v=u$, this is $C(n-u+2)-C(n-u)$; if $v=u+2$, it is
$C(u+2)-C(u)$. In either case the induction hypothesis applies at a
positive starting index smaller than $n$.

If one parity subsequence were bounded, it would be eventually constant.
Its fixed splitting value in \eqref{eq:C} would then make the other
parity eventually constant as well. Let the eventual odd and even
values be $a,b$, respectively. Both are positive and odd, and the
recurrence at sufficiently large odd and even indices gives
\[
 a=C(b)+b-1\ge b,\qquad b=C(a)+a+1>a,
\]
a contradiction. An unbounded nondecreasing sequence of odd integers
starting at $1$ and taking steps $0$ or $2$ visits every positive odd
integer.
\end{proof}

Set $A(0)=B(0)=0$ and, for $k\ge1$, define
\begin{equation}\label{eq:parities}
 A(k)=\frac{C(2k-1)+1}{2},\qquad B(k)=\frac{C(2k)+1}{2}.
\end{equation}
Thus $A$ records the odd positions and $B$ the even positions, with
$C(2k-1)=2A(k)-1$ and $C(2k)=2B(k)-1$.
Both are \emph{slow}: their successive increments belong to $\{0,1\}$.
They are unbounded and satisfy $0\le A(k),B(k)\le k$.
Moreover, all four maps
\[
 A,\quad B,\quad x\mapsto x-A(x),\quad x\mapsto x-B(x)
\]
are nondecreasing and $1$-Lipschitz. Consequently, compositions of these
maps do not amplify absolute errors, a fact used in Section~\ref{sec:rate}.

Substitution into \eqref{eq:C} gives, for $k\ge2$,
\begin{align}
 A(k)&=A(B(k-1))+B(k-B(k-1))-1,\label{eq:A}\\
 B(k)&=A(A(k))+A(k+1-A(k)).\label{eq:B}
\end{align}
For example, at position $2k-1$ the splitting value is
$C(2k-2)=2B(k-1)-1$. The two arguments in \eqref{eq:C} are therefore
$2B(k-1)-1$ and $2(k-B(k-1))$, giving \eqref{eq:A}.

\section{Pinpointing first passages}\label{sec:firstpassage}
To propagate a table to a new window, we must first locate its left
endpoint $k_0$. Evaluating $A(k_0)$ by \eqref{eq:A} requires
$B(k_0-1)$, just outside that window. We overcome this difficulty by
asking where $A$ first reaches the desired value $v$.

Since $A$ is slow and nondecreasing, at the first index $K_v$ where
$A(K_v)=v$, its preceding value is necessarily $A(K_v-1)=v-1$.
Substituting $A(k-1)=v-1$ into \eqref{eq:B} at a trial index $k-1$
suggests the candidate expression $h_v(k):=A(v-1)+A(k-v+1)$ for
$B(k-1)$. Feeding this candidate into \eqref{eq:A} in turn yields a
test function $\Phi_v(k)$. In the application, this test can be
evaluated from two earlier tables.

\begin{lemma}[First passage]\label{lem:pin}
Let $v\ge3$, and write $K_v=\min\{k:A(k)=v\}$. For $k\ge v$, define
\begin{equation}\label{eq:pin-functions}
 \begin{split}
 h_v(k)&=A(v-1)+A(k-v+1),\\
 \Phi_v(k)&=A(h_v(k))+B(k-h_v(k))-1.
 \end{split}
\end{equation}
Then $\Phi_v$ is nondecreasing and $\Phi_v(K_v)=v$. If $k_0-1\ge v$ and
\[
 \Phi_v(k_0-1)=v-1,\qquad \Phi_v(k_0+1)=v+1,
\]
then
\[
 K_v=k_0,\qquad A(k_0-1)=v-1,\qquad B(k_0-1)=h_v(k_0).
\]
\end{lemma}
\begin{proof}
The bounds $A(u)\le u$ give $0\le h_v(k)\le k$, so all terms are
defined. As $k$ increases by one, $h_v(k)$ increases by zero or one;
the same is true of $k-h_v(k)$. Hence $\Phi_v$ is nondecreasing.

At the first passage, $A(K_v-1)=v-1$. Since $K_v\ge v\ge3$,
equation~\eqref{eq:B} applies at $K_v-1$ and gives
\[
 B(K_v-1)=A(v-1)+A(K_v-v+1)=h_v(K_v).
\]
Equation~\eqref{eq:A} at $K_v$ now gives $\Phi_v(K_v)=v$.
If $K_v\le k_0-1$, monotonicity would imply $\Phi_v(K_v)\le v-1$;
if $K_v\ge k_0+1$, it would imply $\Phi_v(K_v)\ge v+1$.
Both are impossible, so $K_v=k_0$, and the preceding values follow.
\end{proof}

The equality at $K_v$ is the essential point: it locates the first
passage by testing $\Phi_v$ on either side of a proposed index. The
identities with the true sequence values are used only at that first
passage. Once it has been located, we know both $A(k_0)$ and
$B(k_0-1)$ and can compute the rest of the window in order.

\section{Exact values near the reference indices}\label{sec:profiles}
Put $s_j=T_j/2$. Dividing \eqref{eq:T} by two gives
\begin{equation}\label{eq:s}
 \begin{gathered}
 s_0=244,\quad s_1=352,\quad s_2=521,\\
 s_j=s_{j-1}+s_{j-3}+
 \begin{cases}6,&j\text{ even},\\-7,&j\text{ odd}\end{cases}
 \qquad(j\ge3).
 \end{gathered}
\end{equation}
For each $j$, the tables below record the values in a short window
around $s_j$, after subtracting $s_{j-3}$. Their width is chosen so that
every nested call needed at level $j$ can be read in the tables at
levels $j-1$ and $j-3$.

\begin{proposition}\label{prop:profiles}
For every $j\ge3$, put $z=s_{j-3}$. At an even level $j$, one has
\begin{equation}\label{eq:profile-even}
 \begin{array}{c|rrrrrr}
 t&-5&-4&-3&-2&-1&0\\\hline
 A(s_j+t)-z&1&1&2&2&2&2\\
 B(s_j+t)-z&1&2&3&3&3&4
 \end{array}
\end{equation}
and at an odd level $j$, one has
\begin{equation}\label{eq:profile-odd}
 \begin{array}{c|rrrrrrr}
 t&-1&0&1&2&3&4&5\\\hline
 A(s_j+t)-z&-4&-3&-3&-2&-2&-2&-1\\
 B(s_j+t)-z&-3&-2&-2&-2&-2&-2&-2
 \end{array}
\end{equation}
\end{proposition}

Here $t$ is the offset from $s_j$. The parity of the level $j$ selects
the table, whereas its two rows correspond to the odd and even
positions of $C$. We prove the proposition by strong induction on $j$.
Three initial levels are needed because the step uses levels $j-1$
and $j-3$.

\subsection{Starting the induction}
Direct evaluation of \eqref{eq:C} through index $3270$ gives the
following values. Each row is read over the offsets of the appropriate
table, in increasing order.
\begin{center}
\begin{tabular}{@{}rrrl@{}}
\toprule
$j$ & $s_j$ & $s_{j-3}$ & Values on the window\\
\midrule
3 & 758 & 244 & $A:\ 240,241,241,242,242,242,243$\\
  &     &     & $B:\ 241,242,242,242,242,242,242$\\[3pt]
4 & 1116 & 352 & $A:\ 353,353,354,354,354,354$\\
  &      &     & $B:\ 353,354,355,355,355,356$\\[3pt]
5 & 1630 & 521 & $A:\ 517,518,518,519,519,519,520$\\
  &      &     & $B:\ 518,519,519,519,519,519,519$\\
\bottomrule
\end{tabular}
\end{center}
Subtracting $s_{j-3}$ gives \eqref{eq:profile-even}--\eqref{eq:profile-odd}.
The largest parity index is $s_5+5=1635$, corresponding to $C(3270)$.
The program in Appendix~\ref{app:seed} verifies these forty values.

\subsection{Locating the left endpoint}
Assume the proposition holds at levels $3,\ldots,j-1$, where $j\ge6$,
and write $z=s_{j-3}$.

\emph{Even levels.}
The source levels $j-1,j-3$ are odd and use \eqref{eq:profile-odd}.
Equation~\eqref{eq:s} gives
\[
 s_j-z=s_{j-1}+6,\qquad z=s_{j-4}+s_{j-6}-7.
\]
We seek $A(s_j-5)=z+1$. Set $v=z+1$ and $k_0=s_j-5$ in
Lemma~\ref{lem:pin}. For $t=-6,-5,-4$,
\[
 h_v(s_j+t)=A(z)+A(s_{j-1}+t+6).
\]
The first term is read at level $j-3$ and the second at level $j-1$.
Using their known values, followed by the definition of $\Phi_v$, gives
\begin{equation}\label{eq:gate-even}
 \begin{array}{c|ccc}
 t&-6&-5&-4\\\hline
 h_v(s_j+t)&z+1&z+1&z+2\\
 \Phi_v(s_j+t)&z&z+1&z+2
 \end{array}
\end{equation}
For instance, the first column is
\[
 \begin{split}
 h_v(s_j-6)&=(s_{j-6}-3)+(s_{j-4}-3)=z+1,\\
 \Phi_v(s_j-6)&=A(z+1)+B(s_{j-1}-1)-1\\
             &=(s_{j-6}-3)+(s_{j-4}-3)-1=z.
 \end{split}
\]
The offsets used for $h_v$ are $(0,0),(0,1),(0,2)$; those for
$\Phi_v$ are $(1,-1),(1,0),(2,0)$. They all lie between $-1$ and $5$,
so the earlier odd-level tables supply every entry. Since
$k_0-1-v=s_{j-1}-1>0$, the first and last columns of
\eqref{eq:gate-even} locate the passage:
\begin{equation}\label{eq:left-even}
 \min\{k:A(k)=z+1\}=s_j-5,\qquad B(s_j-6)=z+1.
\end{equation}

For example, at $j=6$, $s_6=2394$ and $z=s_3=758$. The initial tables
locate the first occurrence of $759$ at $2389$:
\[
 \begin{split}
 h_{759}(2388)&=A(758)+A(1630)=241+518=759,\\
 \Phi_{759}(2388)&=A(759)+B(1629)-1=241+518-1=758,\\
 h_{759}(2390)&=A(758)+A(1632)=241+519=760,\\
 \Phi_{759}(2390)&=A(760)+B(1630)-1=242+519-1=760.
 \end{split}
\]

\emph{Odd levels.}
Now $j\ge7$ is odd, and the source levels are even. We have
\[
 s_j-z=s_{j-1}-7,\qquad z=s_{j-4}+s_{j-6}+6.
\]
To obtain $A(s_j-1)=z-4$, take $v=z-4$ and $k_0=s_j-1$.
For $t=-2,-1,0$,
\[
 h_v(s_j+t)=A(z-5)+A(s_{j-1}+t-2),
\]
and the earlier even-level tables give
\begin{equation}\label{eq:gate-odd}
 \begin{array}{c|ccc}
 t&-2&-1&0\\\hline
 h_v(s_j+t)&z-4&z-3&z-3\\
 \Phi_v(s_j+t)&z-5&z-4&z-3
 \end{array}
\end{equation}
The offsets used for $h_v$ are $(-5,-4)$, $(-5,-3)$, $(-5,-2)$; those for
$\Phi_v$ are $(-4,-5)$, $(-3,-5)$, $(-3,-4)$, all between $-5$ and $0$.
For example,
\[
 \Phi_v(s_j)=A(z-3)+B(s_{j-1}-4)-1
            =(s_{j-6}+2)+(s_{j-4}+2)-1=z-3.
\]
Here $v\ge3$ and $k_0-1-v=s_{j-1}-5>0$. Lemma~\ref{lem:pin} therefore gives
\begin{equation}\label{eq:left-odd}
 \min\{k:A(k)=z-4\}=s_j-1,\qquad B(s_j-2)=z-3.
\end{equation}

\subsection{Completing the tables}
In either case we now know $A(k_0)$ and $B(k_0-1)$ at the left endpoint.
Apply \eqref{eq:B} and \eqref{eq:A} alternately to obtain
\[
 B(k_0),\quad A(k_0+1),\quad B(k_0+1),\quad A(k_0+2),\quad\ldots.
\]
The values already obtained in the new window determine the arguments
of the next calls. The values at those arguments come from the known
tables at levels $j-3$ and $j-1$.

For clarity, write $k=s_j+t$. If $B(k-1)=z+b$, the two arguments
used to compute $A(k)$ have offsets
\[
 (b,\,6+t-b)\quad\hbox{at an even level},\qquad
 (b,\,t-7-b)\quad\hbox{at an odd level},
\]
relative to $s_{j-3}$ and $s_{j-1}$, respectively. The functions read
there are $A$ and $B$. Once $A(k)=z+a$ is known, the two arguments
used to compute $B(k)$ have offsets
\[
 (a,\,7+t-a)\quad\hbox{at an even level},\qquad
 (a,\,t-6-a)\quad\hbox{at an odd level},
\]
and both calls read $A$. The two source value centers add to $z+7$
at an even level and to $z-6$ at an odd level. Including the $-1$ in
\eqref{eq:A}, the resulting offsets from $z$ are therefore the sums
of the two table entries plus $6,7$ in the even case and minus $7,6$
in the odd case, for $A,B$, respectively.

The following table lists all source offsets. Each pair gives the
offset at level $j-3$ first, then that at level $j-1$; the first
$A$-value was fixed by the first-passage argument.
\begin{center}
\begin{tabular}{@{}rcc@{\qquad\qquad}rcc@{}}
\toprule
\multicolumn{3}{c}{Even level $j$} & \multicolumn{3}{c}{Odd level $j$}\\
$t$ & for $A$ & for $B$ & $t$ & for $A$ & for $B$\\
\midrule
$-5$ & fixed & $(1,1)$ & $-1$ & fixed & $(-4,-3)$\\
$-4$ & $(1,1)$ & $(1,2)$ & $0$ & $(-3,-4)$ & $(-3,-3)$\\
$-3$ & $(2,1)$ & $(2,2)$ & $1$ & $(-2,-4)$ & $(-3,-2)$\\
$-2$ & $(3,1)$ & $(2,3)$ & $2$ & $(-2,-3)$ & $(-2,-2)$\\
$-1$ & $(3,2)$ & $(2,4)$ & $3$ & $(-2,-2)$ & $(-2,-1)$\\
$0$  & $(3,3)$ & $(2,5)$ & $4$ & $(-2,-1)$ & $(-2,0)$\\
     &         &         & $5$ & $(-2,0)$ & $(-1,0)$\\
\bottomrule
\end{tabular}
\end{center}
For example, at the start of an even-level table,
\[
 B(s_j-5)=A(z+1)+A(s_{j-1}+1)
         =(s_{j-6}-3)+(s_{j-4}-3)=z+1.
\]
At the start of an odd-level table,
\[
 B(s_j-1)=A(z-4)+A(s_{j-1}-3)
         =(s_{j-6}+1)+(s_{j-4}+2)=z-3.
\]
All offsets in the even case lie in the earlier odd window
$[-1,5]$, and all offsets in the odd case lie in the earlier even
window $[-5,0]$. Reading these entries in order gives precisely
\eqref{eq:profile-even}--\eqref{eq:profile-odd}, completing the induction
and proving Proposition~\ref{prop:profiles}.

At $t=0$, the proposition gives
\begin{equation}\label{eq:centers}
 \begin{array}{c|cc}
 &A(s_j)-s_{j-3}&B(s_j)-s_{j-3}\\\hline
 j\text{ even}&2&4\\
 j\text{ odd}&-3&-2
 \end{array}
\end{equation}
Since $T_j=2s_j$, these are the exact identities
\begin{align}
 C(T_j)&=T_{j-3}+\begin{cases}7,&j\text{ even},\\-5,&j\text{ odd},\end{cases}\label{eq:even}\\
 C(T_j-1)&=T_{j-3}+\begin{cases}3,&j\text{ even},\\-7,&j\text{ odd},\end{cases}
 \qquad(j\ge3).\label{eq:odd}
\end{align}

\section{The global estimate}\label{sec:rate}
The reference indices are geometrically spaced. To control the values
between them, we expand the two recurrences through many levels and
use a quadratic bound on the resulting weighted averages. The exact
values from Section~\ref{sec:profiles} ensure that every argument in
this expansion remains at the expected scale.

\subsection{Geometric scales and bounded errors}\label{sec:scale}
Observe from \eqref{eq:T} that
$T_j-1-\frac{13}{3}(-1)^j$ satisfies the homogeneous Narayana recurrence
$u_j=u_{j-1}+u_{j-3}$. Its initial values are positive. The characteristic
polynomial $x^3-x^2-1$ has one real root $\psi>1$; the other two roots
are conjugate and have modulus $\psi^{-1/2}<1$. The coefficient of
$\psi^j$ in the solution is positive, since positivity excludes a
negative coefficient and the positive recurrence keeps the terms
bounded away from zero. It follows that, for some $\kappa>0$,
\begin{equation}\label{eq:growth}
 s_j=\kappa\psi^j+O(1).
\end{equation}
The sequence $(s_j)$ is increasing: the initial values are increasing,
and \eqref{eq:s} gives $s_j-s_{j-1}\ge s_{j-3}-7>0$.
In particular, \eqref{eq:even}--\eqref{eq:odd} give a bounded error from
the line $\psi^{-3}n$ on both parities at the selected indices.

By \eqref{eq:centers} and \eqref{eq:s}, for $F=A$ or $B$ and $j\ge3$,
\begin{equation}\label{eq:reference-bounds}
 |F(s_j)-s_{j-3}|\le4,\qquad
 |s_j-F(s_j)-s_{j-1}|\le5.
\end{equation}
Thus either map moves three reference levels down, while either
complementary map moves one level down, with uniformly bounded errors.
These shifts explain the weights $\psi^{-3}$ and $\psi^{-1}$, whose sum is $1$.
The $1$-Lipschitz property ensures that an error already present is
not enlarged by further composition.

We also remove the shifts in \eqref{eq:A}--\eqref{eq:B}, at a bounded cost:
\begin{align}
 -2&\le A(k)-A(B(k))-B(k-B(k))\le0,\label{eq:resA}\\
 0&\le B(k)-A(A(k))-A(k-A(k))\le1\qquad(k\ge0).\label{eq:resB}
\end{align}
For \eqref{eq:resA}, replacing $B(k-1)$ by $B(k)$ leaves error $-1$
if the two values agree. Otherwise they differ by one, and the error is
\[
 -1-\bigl(A(B(k))-A(B(k)-1)\bigr)
    +\bigl(B(k-B(k)+1)-B(k-B(k))\bigr),
\]
which lies in $[-2,0]$. For \eqref{eq:resB}, removing the $+1$ in
\eqref{eq:B} changes one $A$-value by zero or one. The cases $k=0,1$
are checked directly.

\subsection{The recursion tree}\label{sec:tree}
Fix an integer $N\ge1$. A node is labeled by a binary word $w$
recording the successive branch choices. The empty word $\varepsilon$
labels the root. At node $w$, we record an argument $x_w$ and the
function $F_w\in\{A,B\}$ evaluated there. Choose $F_\varepsilon=A$ or
$B$ and put $x_\varepsilon=N$. Equations~\eqref{eq:resA}--\eqref{eq:resB}
suggest the following expansion:
\begin{equation}\label{eq:tree}
 \begin{array}{c|cc}
 F_w&(x_{w0},F_{w0})&(x_{w1},F_{w1})\\\hline
 A&(B(x_w),A)&(x_w-B(x_w),B)\\
 B&(A(x_w),A)&(x_w-A(x_w),A)
 \end{array}
\end{equation}
In each row the two child arguments add to $x_w$, while their evaluated
values add to $F_w(x_w)$ with an error of absolute size at most $2$.
Branch $0$ always leads to $A$, whereas branch $1$ switches $A$ and $B$.
All arguments lie in $\{0,\ldots,N\}$.

Write $|w|_0,|w|_1$ for the numbers of zeros and ones in $w$, and give
the path the weight
\[
 \rho_w=\psi^{-3|w|_0-|w|_1},\qquad \rho_\varepsilon=1.
\]
For the linear model $A(x)=B(x)=\psi^{-3}x$, the argument at $w$ would be
$N\rho_w$. We therefore normalize by this scale:
\begin{equation}\label{eq:XY}
 X_w=\frac{x_w}{N\rho_w},\qquad
 Y_w=\frac{F_w(x_w)}{N\rho_w}.
\end{equation}
At the root, $X_\varepsilon=1$ and $Y_\varepsilon$ is $A(N)/N$ or
$B(N)/N$. Since $x_{w0}+x_{w1}=x_w$, we have the exact relation
\begin{equation}\label{eq:Xmean}
 X_w=\psi^{-3}X_{w0}+\psi^{-1}X_{w1}.
\end{equation}
The corresponding relation for $Y$ has a small error inherited from
\eqref{eq:resA}--\eqref{eq:resB}. The next lemma controls the root error
from a common bound on $X$ and a bound on that one-step error.

\subsection{A finite averaging estimate}\label{sec:averaging}
The decisive feature of \eqref{eq:tree} is that reversing two branch
choices preserves their weight but changes the function reached:
$10$ ends at $A$ and $01$ at $B$. At the latter node, $A$ supplies the
next splitting argument. This lets us compare an $A$-value with a
nearby argument at the same scale. The proof turns that comparison
into a bound by summing squared increments over the tree.

\begin{lemma}[Tree averaging estimate]\label{lem:averaging}
Construct \eqref{eq:tree} to depth $H\ge3$ and use \eqref{eq:XY}.
Suppose that
\[
 0\le X_w\le K\quad(|w|\le H),\qquad
 |Y_w-\psi^{-3}Y_{w0}-\psi^{-1}Y_{w1}|\le\delta\quad(|w|<H)
\]
for constants $K$ and $\delta\ge0$. Then, for either initial choice of
$A$ or $B$ at the root,
\begin{equation}\label{eq:averaging}
 |Y_\varepsilon-\psi^{-3}|
 \le K\sqrt{\frac{51\psi^4+15\psi}{H-2}}+3H\delta.
\end{equation}
\end{lemma}
\begin{proof}
Let $D_w=Y_w-\psi^{-3}X_w$. At the root this is the required error
$Y_\varepsilon-\psi^{-3}$, and \eqref{eq:Xmean} shows that $D$ has the same
averaging error as $Y$. Because $D_w$ satisfies the weighted-mean
relation only up to a residual bounded by $\delta$, its squared
increments do not telescope cleanly across levels. To restore exact
telescoping, we replace $D_w$ by the exact backward average $V_w$ of
its terminal values at depth $H$: set $V_w=D_w$ there and define
$V_w=\psi^{-3}V_{w0}+\psi^{-1}V_{w1}$ upwards.
Since $0\le Y_w\le X_w\le K$, both $|D_w|$ and $|V_w|$ are at most
$K$, and backward induction gives
\begin{equation}\label{eq:compensate}
 |V_w-D_w|\le(H-|w|)\delta.
\end{equation}

For any quantity $Z_w$, write
$\E_\ell Z_w=\sum_{|w|=\ell}\rho_wZ_w$.
Define the mean square increment of $X$ and the mean square of $V$ by
\[
 e_\ell=\E_\ell\!\left[\psi^{-3}(X_{w0}-X_w)^2+\psi^{-1}(X_{w1}-X_w)^2\right],
 \qquad m_\ell=\E_\ell V_w^2.
\]
The exact average relation gives
\[
 \psi^{-3}X_{w0}^2+\psi^{-1}X_{w1}^2-X_w^2
 =\psi^{-3}(X_{w0}-X_w)^2+\psi^{-1}(X_{w1}-X_w)^2.
\]
Averaging and summing over the levels makes the left side telescope.
Applying the same identity to $V$, we obtain
\begin{equation}\label{eq:energy}
 \sum_{\ell=0}^{H-1}e_\ell=\E_H X_w^2-X_\varepsilon^2\le K^2,
 \qquad m_{\ell+1}\ge m_\ell,\qquad m_H\le K^2.
\end{equation}
Equivalently, $X$ and $V$ are bounded martingales along a path with
branch probabilities $\psi^{-3},\psi^{-1}$. In particular, the total mean square
change is bounded independently of $H$.

We now relate the errors $D$ to these increments. The two paths $10$
and $01$ satisfy
\[
 F_{w10}=A,\qquad F_{w01}=B,\qquad
 \rho_{w10}=\rho_{w01}=\psi^{-4}\rho_w.
\]
The weights are unchanged when the order is reversed, whereas the
function labels are different. At $w01$ the next branch $0$ uses $A$
as its splitting map, so $x_{w010}=A(x_{w01})$ and
\[
 \frac{A(x_{w01})}{N\rho_w\psi^{-4}}-\psi^{-3}X_{w01}
 =\psi^{-3}(X_{w010}-X_{w01}).
\]
The function $x\mapsto A(x)-\psi^{-3}x$ is $1$-Lipschitz, since its increment
between $u\le v$ lies between $-\psi^{-3}(v-u)$ and $\psi^{-1}(v-u)$. Comparing it
at $x_{w10}$ and $x_{w01}$, with the same denominator $N\rho_w\psi^{-4}$,
gives
\begin{equation}\label{eq:compare-paths}
 |D_{w10}|\le |X_{w10}-X_{w01}|+\psi^{-3}|X_{w010}-X_{w01}|.
\end{equation}
Thus an error in an evaluated value is bounded by changes of the
normalized arguments.

Insert $w1,w,w0$ between $w10$ and $w01$. The right side of
\eqref{eq:compare-paths} is at most
\[
 \begin{split}
 &|X_{w1}-X_w|+|X_{w10}-X_{w1}|\\
 &\qquad+|X_{w0}-X_w|+|X_{w01}-X_{w0}|
       +\psi^{-3}|X_{w010}-X_{w01}|.
 \end{split}
\]
Its square is at most five times the sum of the five squares.
For $0\le\ell\le H-3$, the weighted means of these squares are bounded,
in the displayed order, by
\[
 \psi e_\ell,\quad \psi^4 e_{\ell+1},\quad
 \psi^3 e_\ell,\quad \psi^4 e_{\ell+1},\quad
 \psi e_{\ell+2}.
\]
For example, choosing the suffix $10$ has weight $\psi^{-4}$, which explains
the factor $\psi^4$ in the second bound. In the last term, the coefficient $\psi^{-6}$
cancels the two prescribed zero-branch weights. Hence
\begin{equation}\label{eq:local-energy}
 \E_\ell D_{w10}^2\le
 5\psi^4\bigl(e_\ell+2e_{\ell+1}+\psi^{-3}e_{\ell+2}\bigr).
\end{equation}

It remains to pass from the selected grandchild $w10$ to its parent.
The two-step variance identity for $V$ is
\[
 m_{\ell+2}-m_\ell
 =\E_\ell\sum_{|u|=2}\rho_u(V_{wu}-V_w)^2.
\]
Keeping the nonnegative term with $u=10$ gives
\[
 \E_\ell(V_{w10}-V_w)^2\le\psi^4(m_{\ell+2}-m_\ell).
\]
Decompose
\[
 V_w=D_{w10}+(V_{w10}-D_{w10})+(V_w-V_{w10}).
\]
By \eqref{eq:compensate}, the middle term has absolute size at most
$H\delta$. The inequality $(a+b+c)^2\le3(a^2+b^2+c^2)$ therefore yields
\[
 m_\ell\le3\E_\ell D_{w10}^2+3H^2\delta^2
          +3\psi^4(m_{\ell+2}-m_\ell).
\]
Sum over $\ell=0,\ldots,H-3$. Each $m_\ell$ is at least
$m_0=V_\varepsilon^2$. From \eqref{eq:energy} and \eqref{eq:local-energy},
\[
 \sum_{\ell=0}^{H-3}\E_\ell D_{w10}^2
 \le5\psi^4(3+\psi^{-3})K^2,\qquad
 \sum_{\ell=0}^{H-3}(m_{\ell+2}-m_\ell)\le2K^2.
\]
For the second inequality, the sum telescopes to
$m_{H-2}+m_{H-1}-m_0-m_1$. Consequently,
\[
 (H-2)V_\varepsilon^2
 \le(51\psi^4+15\psi)K^2+3(H-2)H^2\delta^2.
\]
Taking square roots and using \eqref{eq:compensate} once more at the
root gives \eqref{eq:averaging}, since $1+\sqrt3<3$.
\end{proof}

The factor $H^{-1/2}$ comes from summing a square estimate over $H-2$
levels: a fixed error at the root contributes at every level, while
the total squared change stays bounded. We now use the reference
values to allow a logarithmic choice of $H$.

\subsection{A logarithmic depth}\label{sec:depth}
Fix $J$ by $s_J\le N<s_{J+1}$ and take $H=\lfloor J/6\rfloor$, with
$J$ sufficiently large. For a path $w$, put
\[
 d(w)=3|w|_0+|w|_1,\qquad \rho_w=\psi^{-d(w)}.
\]
This counts the reference levels lost along the path. Since
$d(w)\le3H\le J/2$, all reference levels used below are at least $3$.
Monotonicity, the Lipschitz bounds, and \eqref{eq:reference-bounds} give
by induction along a path
\begin{equation}\label{eq:path-bound}
 x_w\le s_{J+1-d(w)}+5|w|\qquad(|w|\le H).
\end{equation}
Indeed, if the current argument is at most $s_i+e$, either branch-$0$
map is at most $s_{i-3}+e+5$, and either complementary map is at most
$s_{i-1}+e+5$. Each step adds at most $5$ to the previous error.

Dividing \eqref{eq:path-bound} by $N\rho_w$ and using
\eqref{eq:growth}, uniformly in $w$,
\[
 0\le X_w\le
 \frac{\kappa\psi^{J+1}+O((H+1)\psi^{3H})}{N}\le K
\]
for a fixed constant $K$. Here $N\asymp\psi^J$, whereas $3H\le J/2$.
Thus the first hypothesis of Lemma~\ref{lem:averaging} holds.
For the second, \eqref{eq:resA}--\eqref{eq:resB} give
\[
 |Y_w-\psi^{-3}Y_{w0}-\psi^{-1}Y_{w1}|\le\frac2{N\rho_w}
                         \le\frac{2\psi^{3H}}N.
\]
We may therefore take $\delta=2\psi^{3H}/N$. For either root function
$F=A$ or $B$, the lemma gives
\[
 \left|\frac{F(N)}N-\psi^{-3}\right|
 \le K\sqrt{\frac{51\psi^4+15\psi}{H-2}}+\frac{6H\psi^{3H}}N
 =O(J^{-1/2}).
\]
The second term is $O(J\psi^{-J/2})$, and
$\log N=J\log\psi+O(1)$. Hence
\[
 A(N)=\psi^{-3}N+O\!\left(\frac{N}{\sqrt{\log N}}\right),\qquad
 B(N)=\psi^{-3}N+O\!\left(\frac{N}{\sqrt{\log N}}\right).
\]
Returning to $C(2N)=2B(N)-1$ and $C(2N-1)=2A(N)-1$ proves
Theorem~\ref{thm:main}.

\section{Further questions}\label{sec:questions}
Is the order $n/\sqrt{\log n}$ in Theorem~\ref{thm:main} optimal?
What is the corresponding normalized limsup? On the upper side of
the limiting line, we conjecture a bounded error.
\begin{conjecture}\label{conj:upper}
For every $n\ge111$,
\[
 C(n)<\frac{n}{\psi^3}+2.
\]
\end{conjecture}
A combinatorial description of the intervals between the reference
points could explain the shapes of the arches and help determine the
optimal order of the error.

\section*{Acknowledgments}
The author thanks Jeffrey Shallit for discussions and automata computations.
The repeating patterns in an automaton he supplied suggested the search
for the local tables used in the proof.

\appendix
\Needspace{30\baselineskip}
\section{The finite initial calculation}\label{app:seed}
The following Python 3 program verifies the three initial tables in
Proposition~\ref{prop:profiles} in integer arithmetic. The lists
\texttt{aa} and \texttt{bb} contain the two rows for even and odd levels,
and \texttt{lo} their first offsets.
\begin{verbatim}
C = [0, 1, 1]
for n in range(3, 3271):
    u = C[n-1]
    assert 1 <= u < n and 1 <= n-u < n
    C.append(C[u] + C[n-u] + (1 if n % 2 == 0 else -1))
    assert C[n] > 0 and C[n] % 2 == 1
    assert C[n] - C[n-2] in (0, 2)

s = [244, 352, 521, 758, 1116, 1630]
lo = [-5, -1]
aa = [[1,1,2,2,2,2], [-4,-3,-3,-2,-2,-2,-1]]
bb = [[1,2,3,3,3,4], [-3,-2,-2,-2,-2,-2,-2]]
for j in (3, 4, 5):
    e = j % 2
    for i in range(len(aa[e])):
        k = s[j] + lo[e] + i
        assert (C[2*k-1]+1)//2 == s[j-3] + aa[e][i]
        assert (C[2*k]+1)//2 == s[j-3] + bb[e][i]
print("PASS: all three initial tables")
\end{verbatim}

\end{document}